\documentclass[11pt,twoside,reqno]{amsart}
\usepackage{amsxtra,amscd}
\usepackage{graphicx}
\usepackage{amsmath}
\usepackage{amsfonts}
\usepackage{amssymb}

\newtheorem{theorem}{Theorem}[section]
\newtheorem{lemma}[theorem]{Lemma}
\theoremstyle{definition}
\newtheorem{definition}[theorem]{Definition}
\newtheorem{example}[theorem]{Example}

\newtheorem{proposition}[theorem]{Proposition}
\newtheorem{remark}[theorem]{Remark}

\newtheorem{corollary}[theorem]{Corollary}

\newtheorem{claim}[theorem]{Claim}

\numberwithin{equation}{section}

\begin{document}
\date{}
\title{The Combinatorial Norm of a Morphism of Schemes}
\author{Feng-Wen An}
\address{School of Mathematics and Statistics, Wuhan University, Wuhan,
Hubei 430072, People's Republic of China}

\address{\textsl{and}}

\address{Faculty of Mathematics and Computer Science, Hubei University, Wuhan,
Hubei 430062, People's Republic of China}

\email{fwan@amss.ac.cn} 

\begin{abstract}
In this paper we will prove that there exists a covariant functor
from the category of schemes to the category of graphs. This functor
provides a combination between algebraic varieties and combinatorial
graphs so that the invariants defined on graphs can be introduced to
algebraic varieties in a natural manner. By the functor, we will
define the combinatorial norm of a morphism of schemes. Then we will
obtain some properties of morphisms of norm not great than one. The
topics discussed here can be applied to study the discrete Morse
theory on arithmetic schemes and Kontsevich's theory of graph
homology.
\end{abstract}

\maketitle

\section*{Introduction}

In this paper we will demonstrate a type of combinatorial properties
of algebraic varieties from the viewpoint of graph theory.

In \S 1 we will first prove that there exists a covariant functor
$\Gamma $, called the graph functor, from the category $Sch$ of
schemes to the category $Grph$ of graphs. Here the trick is
fortunately built on Weil's specializations$^{[25]}$.

This functor $\Gamma $ provides a combination between algebraic
varieties and combinatorial graphs (See \emph{Theorem 1.6}). By the
graph functor $\Gamma$, to each scheme $X$, assign a combinatorial
graph $\Gamma(X)$; to each morphism $f:X \rightarrow Y$ of schemes,
assign a homomorphism $\Gamma(f):\Gamma(X) \rightarrow \Gamma(Y)$ of
combinatorial graphs.

For example, it is easily seen that the combinatorial graph
$\Gamma(Spec(\mathbb{Z}))$ of $Spec(\mathbb{Z})$ is a tree$^{[9]}$,
i.e., a star-shaped graph with the generic point as the center.
Thus, the illustrations of $Spec\left( \mathbb{Z}\right) $ in
[19,21] are not \textquotedblleft correct\textquotedblright .

The practical application of the graph functor $\Gamma$ is that in a
natural manner we can exactly introduce into algebraic varieties the
invariants that are defined on combinatorial graphs and have been
studied in recent decades such as discrete Morse
theory$^{[3-8,14,15,20,23]}$ and graph homology$^{[13,16-18]}$.

However, in general, the combinatorial graphs of most schemes are
not finite; some typical schemes rising from arithmetics (for
example, see [11]) are of infinite dimensions and hence their graphs
are infinite. It follows that there is a situation in which it is
necessary for one to set up some combinatorial quantity for a
morphism of schemes to describe its some property and from its local
data to attempt to obtain its global behavior.

Thus, in \S 2 we will introduce the definition for a combinatorial
norm of a morphism of schemes. The norm of a morphism is defined by
the graph functor $\Gamma$ in an evident manner (See
\emph{Definition 2.1}). The range of norms of morphisms can be any
non-negative integers. For example, a morphism from a scheme to a
zero-dimensional scheme has norm zero; an isomorphism of schemes has
norm one; a length-preserving morphism has norm one; an injective
morphism can have a norm of more than one (See \emph{Example 2.2},
\emph{Remark 2.5}, and \emph{Corollary 2.10}).

We will then conduct an extensive study on some particular types of
 morphisms of schemes by means of combinatorial graphs and their
 lengths.
 As a common characteristic, all these morphisms have
the norms of not great than one. In \emph{Proposition 2.6} we will
give an application of the length-preserving morphism.

It is easily seen that \textquotedblleft injective
$\nLeftrightarrow$ length-preserving\textquotedblright and that
\textquotedblleft norm one $\nLeftrightarrow$
injective\textquotedblright for a morphism of schemes (See
\emph{Remark 2.5}). So in \emph{Theorem 2.7} we will give a
sufficient condition to a morphism of schemes whose norm is not
greater than one. And in \emph{Theorem 2.8} we will obtain a
comparison between injective and length-preserving morphisms of
schemes.

The results on norms of morphisms between schemes, discussed in the
paper, can be applied to topics on the discrete Morse theory on
arithmetic schemes and Kontsevich's theory of graph homology [for
example, see our subsequent paper].

Finally, I would like to express my sincere gratitude to Professor
Li Banghe for his invaluable advice on algebraic geometry and
topology. I am also indebted to Dr Yuji Odaka (Tokyo University) for
pointing out an error in an earlier version of the preprint of the
paper.

\section{The Combinatorial Graph Of A Scheme}

\subsection{Notation}

Let us recall some definitions in [10,25]. Let $E$ be a topological
space $E$ and $x,y\in E$. If $y$ is in the closure $\overline{\{x\}}$, $y$ is a \textbf{%
specialization} of $x$ (or, $x$ is said to be a \textbf{generalization} of $%
y $) in $E$, denoted by $x\rightarrow y$. Put $Sp\left( x\right)
=\{y\in E\mid x\rightarrow y\}$.
 It is evident that $Sp\left( x\right) =\overline{\{x\}}$ is an irreducible closed
 subset in $E$.

 If $x\rightarrow
y$ and $y\rightarrow x$ both hold in $E$, $y$ is a \textbf{generic
specialization} of $x$ in $E$, denoted by $x\leftrightarrow y$. The point $x$
is \textbf{generic} (or \textbf{initial}) in $E$ if we have $%
x\leftrightarrow z$ for any $z\in E$ such that $z\rightarrow x$. And $x$ is
\textbf{closed} (or \textbf{final}) if we have $x\leftrightarrow z$ for any $%
z\in E$ such that $x\rightarrow z.$ We say that $y$ is a
\textbf{closest specialization}
of $x$ in $X$ if either $z=x$ or $z=y$ holds for any $z\in X$ such that $%
x\rightarrow z$ and $z\rightarrow y$.

\subsection{Any specialization is contained in an affine open set}

Let $E=Spec\left( A\right) $ be an affine scheme. For any point
$z\in Spec\left( A\right)$, denote by $j_{z}$ the corresponding
prime ideal in $A$. It is clear that there is a specialization
$x\rightarrow y$ in $Spec\left( A\right) $ if and only if
$j_{x}\subseteq j_{y}$ in $A$. It follows that there is a generic
specialization $x\leftrightarrow y$ in $Spec\left( A\right) $ if and
only if $x=y$. Now given a scheme $X$.

\begin{lemma}
Let $x,y \in X$. Then we have $x\leftrightarrow y$ in $X$ if and
only if $x=y$.
\end{lemma}

\begin{proof}
$\Leftarrow$. Trivial. Prove $\Rightarrow$. Assume $x\leftrightarrow
y$ in $X$. Let $U$ be an affine open set of $X$ containing $x$.
From $x\leftrightarrow y$ in $X$, we have ${Sp}(x)={Sp}(y)$; then $x\in{Sp}%
(x)\bigcap U={Sp}(y)\bigcap U \ni y$, that is, $y$ is contained in
$U$. Hence, $x\leftrightarrow y $
in $U$. It follows that $x=y$ holds in $U$ (and of course in $%
X$).
\end{proof}

\begin{lemma}
Any specialization $x\rightarrow y$ in $X$ is contained in some
affine open subset $U$ of $X$, that is, the points $x,y$ are both
contained in $U$. In particular, each affine open set of $X$
containing $y$ must contain $x$.
\end{lemma}

\begin{proof}
Take a specialization $x\rightarrow y$ in $X$ with $x \not= y$. Then
$y$ is a limit point of the one-point set $\{x\}$ since $y$ is
contained in the topological closure $Sp(x)$ of $\{x\}$. Let
$U\subseteq X$ be an open set containing $y$. We have $U \bigcap
(\{x\}\setminus \{y\})\neq \emptyset$ by the definition for a limit
point of a set (see any standard textbook for general topology). We
choose $U$ to be an affine open set of $X$.
\end{proof}

\subsection{Any morphism preserves specializations}

Let $E$ be a topological space and let $IP(W)$ be the set of the
generic points in a subset $W$ of $E$.

$E$ is said to be of the $\left( UIP\right)-$\textbf{property} if
$E$ satisfies the conditions:

$(i)$  $IP(W)$ is a nonvoid set for any nonvoid irreducible closed
subset $W$ of $E$; $(ii)$ for any irreducible closed subset $V$ and
$W$ of $E$ with $V\neq W$, there is $x_{V}\neq x_{W}$ for any
$x_{V}\in IP(V)$ and any $x_{W}\in IP(W)$.

Let $f:E\rightarrow F$ be a map of spaces. Then
$f$ is said to be $IP-$\textbf{preserving} if we have $%
f\left( x_{0}\right)\in IP(\overline{f(U)}) $ for any closed subset
$U$ of $E$ and any $x_{0}\in IP(U)$. Here $\overline{f(U)}$ denotes
the topological closure of the set ${f\left( U\right) }$.

The map $f$ is said to be \textbf{%
specialization-preserving} if there is a specialization $f\left(
x\right) \rightarrow f\left( y\right) $ in $F$ for any
specialization $x\rightarrow y$ in $E$.

\begin{remark}
By Zorn's Lemma it is seen that any irreducible $T_{0}-$spaces have
the $\left( UIP\right) -$property if there are generic points. In
particular, any scheme is of the $(UIP)-$property.
\end{remark}

\begin{proposition}
Let $f:E\rightarrow F$ be a continuous map of topological spaces.

$\left( i\right) $ $f$ is specialization-preserving if and only if $f$ is $%
IP-$preserving.

$\left( ii\right) $ Let $F$ be of the $\left( UIP\right)-$property.
Then $f$ is specialization-preserving.
\end{proposition}

\begin{proof}
$\left( i\right) $ $\Rightarrow$. Let $f$ be
specialization-preserving. Take any closed subset $U$ of $E$ and any
$x_{0}\in IP(U)$. Without loss of generality, we assume that $U$ is
irreducible.

For any $x\in U$, there is a specialization $x_{0}\rightarrow x$ in
$U$. From the assumption we have a specialization
$f(x_{0})\rightarrow f(x)$ in $f(U)$; then $f(x_{0})\rightarrow y$
in $f(U)$ holds for any $y\in f(U)$. Put
$$Sp(f(x_{0}))\mid _{f(U)}=\{z\in f(U)\mid f(x_{0})\rightarrow z
\text{ in } f(U)\}.$$

As $Sp(f(x_{0}))\mid _{f(U)}=f(U)$, we have
$$Sp(f(x_{0}))=\overline{Sp(f(x_{0}))\mid
_{f(U)}}=\overline{f(U)}.$$
It follows that for any $z\in \overline{f(U)}$ there is a
specialization $f(x_{0})\rightarrow z$ in $\overline{f(U)}$. Hence,
$$f(x_{0})\in IP(\overline{f(U)}).$$

$\Leftarrow$. Let $f$ be IP-preserving. Take any specialization
$x_{0}\rightarrow x$ in $E$. Let $U=Sp(x_{0})$. We have
$f(x_{0}),f(x) \in f(U)$; then $$f(x)\in
\overline{f(U)}=Sp(f(x_{0}));$$ hence there is a specialization
$f(x_{0})\rightarrow f(x)$ in $F$.

$\left( ii\right) $ Fixed any specialization $x\rightarrow y$ in $E$%
. From the irreducibility of $Sp\left( x\right)$, it is seen that
$\overline{ f\left( Sp\left( x\right) \right) }$ is an irreducible
closed subset in $F$. As $f\left( x\right) \in f\left( Sp\left(
x\right) \right)$, there is $$ Sp\left( f\left( x\right) \right)
\subseteq\overline{f\left( Sp\left( x\right) \right) };$$ as $F$ has
the $\left( UIP\right) -$property, it is seen that $$
\overline{f\left( Sp\left( x\right) \right) }=Sp\left( f\left(
x\right) \right) $$ holds since they both contain $f(x)$ as a
generic point. Similarly, we have $$ \overline{f\left( Sp\left(
y\right) \right) }=Sp\left( f\left( y\right) \right).$$

As $Sp\left( x\right) \supseteq Sp\left( y\right) ,$ we have $
f\left( Sp\left( x\right) \right) \supseteq f\left( Sp\left(
y\right) \right)$; then $ Sp\left( f\left( x\right) \right)
\supseteq Sp\left( f\left( y\right) \right)$. So there is a
specialization $f\left( x\right) \rightarrow f\left( y\right) $ in
$F$.
\end{proof}

For the case of schemes, we have the following lemma.

\begin{lemma}
Any morphism of schemes is specialization-preserving.
\end{lemma}

\begin{proof}
It is immediate from \emph{Remark 1.3} and \emph{Lemma 1.4}.
\end{proof}

\subsection{The graph functor $\Gamma$ from schemes to graphs}

Now we have such a covariant functor from the category of schemes to
the category of graphs.

\begin{theorem}
There exists a covariant functor $\Gamma$ from the category $Sch$ of
schemes to the category $Grph$ of graphs defined in such a natural
manner:

$(i)$ To each scheme $X$, assign the graph $\Gamma(X)$ in which the
vertex set is the set of points in the underlying space $X$ and the
edge set is the set of specializations in $X$.

Here, for any points $x,y \in X$, we say that there is an edge from
$x$ to $y$ if and only if there is a specialization $x\rightarrow y$
in $X$.

$(ii)$ To each morphism $f:X\rightarrow Y$ of schemes, assign the
homomorphism $\Gamma(f): \Gamma(X)\rightarrow \Gamma(Y)$ of graphs
such that any specialization $x \rightarrow y$ in the scheme $X$ as
an edge in $\Gamma(X)$, is mapped by $\Gamma(f)$ into the
specialization $f(x)\rightarrow f(y)$ as an edge in $\Gamma(Y)$.
\end{theorem}

\begin{proof}
It is immediate from \emph{Lemmas 1.1} and \emph{1.5}.
\end{proof}

The above functor $\Gamma$ from the category $Sch$ of schemes to the
category $Grph$ of graphs is said to be the \textbf{graph functor}
in the paper. For a scheme $X$, the graph $\Gamma(X)$ is said to the
(\textbf{associated}) \textbf{graph} of $X$; for a morphism
$f:X\rightarrow Y$ of schemes, the graph homomorphism $\Gamma(f)$ is
said to be the (\textbf{associated}) \textbf{homomorphism} of $f$.

In general, the graph $\Gamma(X)$ of a scheme $X$ is not a finite
graph. For example, the graph $\Gamma(Spec(\mathbb{Z}))$ of
$Spec(\mathbb{Z})$ is a star-shaped graph. The graph
$\Gamma(Spec(\mathbb{Z}[t]))$ of $Spec(\mathbb{Z}[t])$ is a graph of
infinitely many loops.

\begin{remark}
By the graph functor $\Gamma$, many invariants defined on graphs can
be introduced into algebraic varieties in a natural manner, such as
the discrete Morse theory$^{[3-8,15,23]}$ and the Kontsevich's graph
homology$^{[13,16-18]}$.

The graph functor $\Gamma$ can provide us some type of completions
of birational maps between algebraic varieties.
\end{remark}

\section{The Combinatorial Norm of a Morphism}

In this section the graph functor $\Gamma$ will be applied to set up
a combinatorial quantity for a morphism of schemes, called the norm
of a morphism. For convenience, in the following we will identify a
scheme $X$ with its graph $\Gamma(X)$ and identify a specialization
$X$ with its edge in $\Gamma(X)$.

Notice that here a scheme is not necessarily finite-dimensional
except when otherwise specified.

\subsection{Definition and notation}

By the graph functor $\Gamma$ we have the notion of combinatorial
quantities in a scheme which we borrow from graph theory. Let $X$ be
a scheme. Fixed a specialization $x\rightarrow y$ in $X$ (regarded
as an edge in the graph $\Gamma(X)$).

By a \textbf{restrict chain of specializations }$\Delta \left(
x,y\right) $ (of \textbf{length} $n$) from $x$ to $y$ in $X$, we
understand a chain of specializations
\begin{equation*}
x=x_{0}\rightarrow x_{1}\rightarrow \cdots \rightarrow x_{n}=y
\end{equation*}%
in $X$, where $x_{i}\neq x_{i+1}$ and each $x_{i+1}$ is a
specialization of $x_{i}$ for $0 \leq i \leq n-1$.

The \textbf{length} $l\left( x,y\right) $ of the specialization
$x\rightarrow y$ is the supremum among all the lengths of restrict
chain of specializations from $x$ to $y$. Let $W$ be a subset of
$X$. The \textbf{length} $l\left( W\right)$ of $W$ is defined to be
\begin{equation*}
\sup \{l\left( x,y\right) \mid\text{there is a specialization }%
x\rightarrow y \text{ in }W\}.
\end{equation*}
In particular, the \textbf{length} $l(x)$ of a point $x\in X$ is
defined to be the length of the subspace ${Sp}(x)$ in $X$.

Let $l(W)<\infty$. A restrict chain $\Delta$ of specializations in $W$ is a \textbf{%
presentation} for the length $l(W)$ of $W$ if the length of $\Delta$ is equal to $%
l\left( W\right) .$

Let $\dim X < \infty$. We have $l\left( X\right)=\dim X$. Moreover,
let $\Delta\left( x_{0},x_{n}\right) $ be a presentation for the
length of $X$. Then $x_{0}$ is generic and $x_{n}$ is closed in $X$.

\subsection{The norm of a morphism between schemes}

By the graph functor $\Gamma$ we can define the combinatorial norm
of a morphism between schemes.

\begin{definition}
Let $f:X\rightarrow Y$ be a morphism of schemes.

$\left( i\right) $ $f$ is said to be \textbf{bounded} if there
exists a constant $\beta\in \mathbb{R}$ such that
\begin{equation*}
l\left( f\left( x_{1}\right) ,f\left( x_{2}\right) \right) \leq\beta \cdot
l\left( x_{1},x_{2}\right)
\end{equation*}
holds for any specialization $x_{1}\rightarrow x_{2}$ in $X$ with
$l\left( x_{1},x_{2}\right) <\infty$.

$\left( ii\right) $ Let $f$ be bounded. If $\dim X=0$, define $\left\Vert
f\right\Vert =0$; if $\dim X>0$, define
\begin{equation*}
\left\Vert f\right\Vert =\sup\{\frac{l\left( f\left( x_{1}\right) ,f\left(
x_{2}\right) \right) }{l\left( x_{1},x_{2}\right) }:x_{1}\rightarrow x_{2}%
\text{\emph{\ in }}X,\text{ }0<l\left( x_{1},x_{2}\right) <\infty\}.
\end{equation*}
Then the number $\left\Vert f\right\Vert $ is said to be the \textbf{norm}
of $f$.
\end{definition}

\begin{example}
The norm of a morphism of schemes can be equal to any non-negative
integer.

$\left( i\right) $ The $k-$rational points of a $k-$variety are morphisms of
norm zero.

$\left( ii\right) $ Let $s,t$ be variables over a field $k$ and let
\begin{equation*}
f:Spec\left( k\left[ s,t\right] \right) \rightarrow Spec\left( k\left[ t%
\right] \right)
\end{equation*}
be induced from the evident embedding of $k-$algebras. Then $\left\Vert
f\right\Vert =1.$

$\left( iii\right) $ Let $t$ be a variable over $\mathbb{Q}$ and let
\begin{equation*}
f:Spec\left( k\left[ t\right] \right) \rightarrow Spec\left( \mathbb{Z}\left[
t\right] \right)
\end{equation*}
be induced from the evident embedding. Then $\left\Vert f\right\Vert =2.$
\end{example}

Take a scheme $X$. Two points $x$ and $y$ in $X$ are said to be
$Sp-$\textbf{connected} if either $x\rightarrow y$ or $y\rightarrow
x$
holds in $X$. Otherwise, $x$ and $y$ are said to be $Sp-$\textbf{disconnected%
} if they are not $Sp-$connected.

A nonvoid subset $A$ in $X$ is said to be $Sp-$\textbf{connected} if
any two elements in $A$ are $Sp-$connected.

By the norm of a morphism and the graph functor $\Gamma$ it is seen
that there are two specified types of data, the latitudinal data and
the longitudinal data, for us to describe a morphisms of scheme
(such as \emph{Remarks 2.3-4}).

\begin{remark}
\textbf{(The Latitudinal Data).} Let $f:X\rightarrow Y$ be a
morphism of schemes. There are some cases for the latitudinal data
such as the following for one to describe $f$:

$\left( i\right) $ $f$ is said to be \textbf{level-separated} if the points $%
f\left( x\right) $ and $f\left( y\right) $ are $Sp-$disconnected in
$Y$ for any $x,y\in X$ that are $Sp-$disconnected and of the same
lengths.

$\left( ii\right) $ $f$ is said to be \textbf{level-reduced} if
$f\left( x\right) $ and $f\left( y\right) $ are $Sp-$connected in
$Y$ for any $x,y\in X$ that are $Sp-$disconnected and of the same
lengths.

$\left( iii\right) $ $f$ is said to be \textbf{level-mixed} if $f$ is
neither level-separated nor level-reduced.
\end{remark}

\begin{remark}
\textbf{(The Longitudinal Data).} Let $f:X\rightarrow Y$ be a
morphism of schemes. There are some cases for the longitudinal data
such as the following for one to describe $f$:

$\left(i\right) $ $f$ is said to be \textbf{null} if $\left\Vert
f\right\Vert =0.$

$\left( ii\right) $ $f$ is said to be \textbf{asymptotic} if
$\left\Vert f\right\Vert =1$.

$\left(iii\right) $ $f$ is said to be \textbf{length-preserving} if
$l\left( f\left( x\right) ,f\left( y\right) \right) =h$ holds for
any specialization $x\rightarrow y$ in $X$ such that $l\left(
x,y\right)=h<\infty$.
\end{remark}

\begin{remark}
Let $f:X\rightarrow Y$ be a morphism of schemes.

$\left( i\right) $ Let $1\leq\dim X=\dim Y<\infty $. It is seen that
$\left\Vert f\right\Vert \geq 1$ if $f$ is surjective.

$\left( ii\right) $ Let $f$ be length-preserving. Then $\dim X
\leqslant \dim Y$ and $\left\Vert f\right\Vert =1$ hold. In general,
it is not true that $f$ is injective.

Conversely, let $f$ be injective. In general, it is not true that
$f$ is length-preserving.

$\left( iii\right) $ Let $\left\Vert f\right\Vert =1$. In general,
it is not true that $f$ is injective.

Conversely, let $f$ be injective. In general, it is not true that
$\left\Vert f\right\Vert =1$ holds (See \emph{Corollary 2.11}).
\end{remark}

In the following there will be an extensive study on several
particular morphisms between schemes by means of combinatorial
graphs and their lengths.
 As a common characteristic, all these morphisms have
the norms of not great than one.

\subsection{An application of a length-preserving morphism}

A morphism $f:X\rightarrow Y$ of schemes is said to be of \textbf{finite }$%
J- $\textbf{type} if $f$ is of finite type and the homomorphism $$
f^{\sharp}\mid_{V}:\mathcal{O}_{Y}\left( V\right) \rightarrow f_{\ast}%
\mathcal{O}_{X}\left( U\right)
$$
of rings is of $J-$type for any affine open sets $V$ of $Y$ and $U$ of $%
f^{-1}\left( V\right)$.

Here, a homomorphism $\tau:R\rightarrow S$ of commutative rings is
said to be of $J-$\textbf{type} if there is an identity $$
\tau^{-1}\left( \tau\left( I\right) S\right) =I $$  for every prime
ideal $I$ in $R.$

\begin{proposition}
Let $f:X\rightarrow Y$ be a morphism between irreducible schemes. Suppose $%
\dim X<\infty $. Then we have
\begin{equation*}
\dim X=\dim Y<\infty
\end{equation*}%
if $f$ is length-preserving and of finite $J-$type.
\end{proposition}

\begin{proof}
Let $f$ be length-preserving and of finite $J-$type. We have
$$
l\left( X\right) =l\left( f\left( X\right) \right) \leq l\left( Y\right) .
$$
It is seen that $ \dim X\leq\dim Y $ holds since $ \dim X=l\left(
X\right)$ and $\left( Y\right) =\dim Y $ hold.

Let $x\in X$ and $y=f\left( x\right) \in Y.$ As $f$ is of finite $J-$%
type, there are affine open subsets $V$ of $Y$ and $U$ of $f^{-1}\left(
V\right) $ such that
$$
f^{\sharp}\mid_{V}:\mathcal{O}_{Y}\left( V\right) \rightarrow f_{\ast}%
\mathcal{O}_{X}\left( U\right)
$$
is a homomorphism of $J-$type.

Let $V=Spec\left( R\right) $ and $U=Spec\left( S\right)$. We have
$\dim U=\dim X$ and $\dim V=\dim Y$.

Take any restrict chain of specializations
$$
y_{0}\rightarrow y_{1}\rightarrow\cdots\rightarrow y_{n}
$$
in $V.$ By \S 1.2 we obtain a chain of prime ideals
$$
j_{y_{0}}\subsetneqq j_{y_{1}}\subsetneqq\cdots\subsetneqq j_{y_{n}}
$$
in $R,$ where each $j_{y_{i}}$ denotes the prime ideal in $R$ corresponding to the point $%
y_{i}$ in $V$.

By \emph{Corollary 2.3}$^{[22]}$ it is seen that there are a chain
of prime ideals $$ I_{0}\subseteq I_{1}\subseteq\cdots\subseteq
I_{n}
$$ in $S$ such that $$ f^{\#-1}\left( I_{i}\right) =j_{y_{i}}.
$$ It follows that there is a restrict chain of specializations $$
x_{0}\rightarrow x_{1}\rightarrow\cdots\rightarrow x_{n} $$ in $U$
such that $f\left( x_{i}\right) =y_{i}$ and $j_{x_{i}}=I_{i}$.
Hence, $ l\left( U\right) \geq l\left( V\right) $ holds.

It is evident that $$ l(U)=\dim U\text{ and }l(V)=\dim V $$ hold for
the subspaces $U,V$. So we must have $$ \infty>\dim X\geq\dim Y.
$$ This completes the proof.
\end{proof}

\subsection{A sufficient condition to a morphism of norm not
greater than one}

Let $f:X\rightarrow Y$ be a morphism of schemes. Fixed a point
$x_{0}\in X$. Then $f$ is said to be \textbf{Sp-connected} at
$x_{0}$ if the pre-image $f^{-1}\left( Sp\left( f\left(x_{0}\right)
\right) \right)$ is a $Sp-$connected set. And $f$ is said to be
\textbf{Sp-proper} at $x_{0}$ if the pre-image $f^{-1}(Sp\left(
f\left( x_{0}\right) \right) )$ is equal to $Sp\left( x_{0}\right)$.

The morphism $f$ is said to be of \textbf{Sp-type} on $X$ if $f$ is
either Sp-connected or Sp-proper at each point $x\in X$.

In fact, such a datum locally defined by specializations can control
the global behavior of a morphism of schemes$^{[11]}$. For example,
a morphism induced by a homomorphism of Dedekind domains of schemes
is of Sp-type. An isomorphism of schemes is of Sp-type; the converse
is not true.

Here we give a sufficient condition to a morphism of norm not
greater than one.

\begin{theorem}
Let $f:X\rightarrow Y$ be a morphism of schemes. Then we have
\begin{equation*}
0\leq \left\Vert f\right\Vert \leq 1
\end{equation*}%
if $f$ is of Sp-type on $X$.
\end{theorem}

\begin{proof}
If $\dim X=0$ or $\dim Y=0$ we have $\left\Vert f\right\Vert =0$. In
the following we assume $\dim X>0$ and $\dim Y>0.$

Fixed any specialization $x_{1}\rightarrow x_{2}$ in $X$ such that
$$ 0<\left( x_{1},x_{2}\right) <\infty$$ and $$l\left( f\left(
x_{1}\right) ,f\left( x_{2}\right) \right) >0.$$ We will proceed in
two steps.

\textsl{\textbf{Step 1.}} Let $x_{1}\rightarrow x_{2}$ in $X$ be a
closest specialization. In the following we will prove that the
specialization $f\left( x_{1}\right) \rightarrow f\left(
x_{2}\right) $ in $Y$ is also closest.

Hypothesize that the specialization $f\left( x_{1}\right)
\rightarrow f\left( x_{2}\right) $ in $Y$ is not a closest one. It
follows that for the length we have $$ l\left( f\left( x_{1}\right)
,f\left( x_{2}\right) \right) \geq2 .$$ Then $$ Sp\left( f\left(
x_{1}\right) \right) \supsetneqq Sp\left( f\left( x_{2}\right)
\right);$$
$$l(Sp\left( f\left( x_{1}\right) \right))=\dim Sp\left( f\left(
x_{1}\right) \right) \geq 2.$$

Take a point $y_{0}\in Y$ such that $$f\left( x_{1}\right) \not
=y_{0}\not =f\left( x_{2}\right)$$ and that there are
specializations
$$
f\left( x_{1}\right) \rightarrow y_{0}\rightarrow f\left(
x_{2}\right)$$ in $Y$. We have
$$
Sp\left( f\left( x_{1}\right) \right) \supsetneqq Sp\left( y_{0}\right)
\supsetneqq Sp\left( f\left( x_{2}\right) \right).
$$

As $f$ is of Sp-type, it is seen that there are two cases for the
point $x_{1}$.

\textsl{\textbf{Case (i).}} Let $f$ be Sp-proper at $x_1$. That is,
$Sp(x_{1})=f^{-1}(Sp(f(x_{1})))$.

Then $f(Sp(x_{1}))=Sp(f(x_{1}))$ holds. As $y_{0}\in Sp(f(x_{1}))$,
it is seen that there is a point $ x_{0}\in Sp(x_{1}) $ such that
$y_{0}=f(x_{0})$. Hence, we obtain a specialization
$x_{1}\rightarrow x_{0}$ in $X$.

Similarly, there are two subcases for the point $x_{0}$ such as the
following:

\textsl{\textbf{Subcase (i$_{a}$).}} Assume
$Sp(x_{0})=f^{-1}(Sp(f(x_{0})))$.

As $f(x_2) \in Sp(f(x_0))$, it is seen that the point $x_2$ is
contained in the set $Sp(x_0)$; then we have a specialization
$x_{0}\rightarrow x_{2}$ in $X$. So there are specializations
$$
x_{1}\rightarrow x_{0}\rightarrow x_{2}
$$
in $X$.

\textsl{\textbf{Subcase (i$_{b}$).}} Assume that
$f^{-1}(Sp(f(x_{0})))$ is a $Sp-$connected set.

Then either $x_{0}\rightarrow x_{2}$ or $x_{2}\rightarrow x_{0}$ is
a specialization in $X$; by Proposition 2.5 it is seen that only
$x_{0}\rightarrow x_{2}$ holds in $X$ since $y_{0}\not= f(x_{2})$
and $y_{0}\rightarrow f(x_{2})$; then there are specializations
\begin{equation*}
x_{1}\rightarrow x_{0}\rightarrow x_{2}
\end{equation*}
in $X$.

\textsl{\textbf{Case (ii).}} Let $f$ be Sp-connected at $x_1$. That
is, $f^{-1}(Sp(f(x_{1})))$ is a $Sp-$connected set.

As $y_{0}\in Sp(f(x_{1}))$, we have $ x_{0}\in Sp(x_{1}) $ with
$y_{0}=f(x_{0})$. As $y_{0}\not= f(x_{1})$, it is seen that there is
a specialization $x_{1}\rightarrow x_{0}$ in $X$. As $x_{2}\in
Sp(x_{1})$ and $y_{0} \neq f(x_2)$, we have a specialization
$x_{0}\rightarrow x_{2}$ in $X$; then we obtain specializations
\begin{equation*}
x_{1}\rightarrow x_{0}\rightarrow x_{2}.
\end{equation*}

From the above cases, we must have $ l\left( x_{1},x_{2}\right) \geq
2. $ Hence, $x_{1}\rightarrow x_{2}$ in $X$ is not a closest
specialization, which is in contradiction to the assumption.
Therefore, $ f(x_{1})\rightarrow f(x_{2}) $ must be a closest
specialization in $Y$.

\textsl{\textbf{Step 2.}} Let $x_{1}\rightarrow x_{2}$ in $X$ be not closest. Put $%
l\left( x_{1},x_{2}\right) =n \geq 2. $ There are the closest
specializations
\begin{equation*}
z_{1}\rightarrow z_{2}\rightarrow\cdots\rightarrow z_{n+1}
\end{equation*}
in $X$ with $z_{1}=x_{1}$ and $z_{n+1}=x_{2}.$

It is seen that either for some $1\leq i\leq n$ there is an identity
$$
f\left( z_{i}\right) =f\left( z_{i+1}\right)
$$
or
$$
f\left( z_{1}\right) \rightarrow f\left( z_{2}\right) \rightarrow
\cdots\rightarrow f\left( z_{n+1}\right)
$$ are a restrict chain of specializations
in $Y.$ By \textsl{Step 1} we have
\begin{equation*}
l\left( f\left( x_{1}\right) ,f\left( x_{2}\right) \right) \leq l\left(
x_{1},x_{2}\right).
\end{equation*}
This proves $ \left\Vert f\right\Vert \leq 1. $
\end{proof}

\subsection{A theorem on the comparison  between injective and
length-preserving morphisms of schemes}

A scheme $X$ is said to be \textbf{cat\'{e}naire} if the underlying
space of $X$ is a cat\'{e}naire space. For cat\'{e}naire spaces, see
\emph{ch 8, \S 1} of [1].

Let $x\rightarrow y\rightarrow z$ be specializations in a
cat\'{e}naire scheme. It is clear that
$$l(x,z)=l(x,y)+l(y,z)$$ holds by definition for cat\'{e}naire
space$^{[1]}$.

Now we obtain a result on the comparison between injective and
length-preserving morphisms of schemes.

\begin{theorem}
Let $f:X\rightarrow Y$ be a morphism of irreducible schemes. Suppose $\dim
Y<\infty $.

Let $f$ be injective and of Sp-type. Then $%
f$ is length-preserving and level-separated.

Conversely, let $X$ be cat\'{e}naire. Then $f$ is injective if $f$
is length-preserving and level-separated.
\end{theorem}

\begin{proof}
\textbf{(i).} Prove the first half of the theorem. Let $f$ be
injective and of Sp-type. We will prove that $f$ is
length-preserving and level-separated.

It is seen that $\dim X<\infty$ holds. Otherwise, hypothesize $\dim
X =\infty$. For any $n \in \mathbb{N}$ we have a chain of
irreducible closed subsets
$$X_{n}\subsetneqq X_{n-1}\subsetneqq \cdots\subsetneqq X_{0}$$ in
$X$. By \emph{Remark 1.3}, there are points $v_{j}\in X_{j}$ such
that $Sp(v_{j})=X_{j}$ for $0\leqslant j \leqslant n$. Then we have
a chain of specializations $$v_{0}\rightarrow v_{1}\rightarrow\cdots
\rightarrow v_{n}$$ in $X$. By \emph{Lemma 1.5} it is seen that
there are specializations $$f(v_{0})\rightarrow
f(v_{1})\rightarrow\cdots \rightarrow f(v_{n})$$ in $Y$. Hence, $n
\leqslant l(Y)=\dim(Y)$, where we will obtain a contradiction.

 As $\dim X<\infty$, we have $l(X)=\dim X$. In the following we will proceed in three steps.

\textsl{\textbf{Step 1.}} Show $f$ is length-preserving. Take a
chain of specializations
\begin{equation*}
z_{0}\rightarrow z_{1}\rightarrow\cdots\rightarrow z_{n}
\end{equation*}
in $X$ such that $l(z_{0},z_{n})=n$. We have specializations
\begin{equation*}
f\left( z_{0}\right) \rightarrow f\left( z_{1}\right) \rightarrow
\cdots\rightarrow f\left( z_{n}\right)
\end{equation*}
in $Y.$

As $f$ is of Sp-type, by \emph{Theorem 2.7} it is seen that $\|f
\|\leq 1$ and then $l(f(z_0),f(z_n))\leq l(z_0,z_n)=n$; as $f$ is
injective, it is seen that $f\left( z_{i}\right) \not =f\left(
z_{j}\right) $ holds for all $i\not =j$; hence we have
\begin{equation*}
l\left( z_{0},z_{n}\right) =l\left( f\left( z_{0}\right) ,f\left(
z_{n}\right) \right)=n.
\end{equation*}

It follows that
\begin{equation*}
l\left( x,y\right) =l\left( f\left( x\right) ,f\left( y\right) \right)
\end{equation*}
holds for any specialization $x\rightarrow y$ of finite length. This
proves that $f$ is length-preserving.

\textsl{\textbf{Step 2.}} Show that $ l(f(z))=l(z)<\infty $ holds
for any $z\in X$.

In fact, by \textsl{Step 1} above we have $l(z)\leq l(X)=\dim X
<\infty$ for any $z\in X$.  As $f$ is length-preserving, we have $
l(z)\leq l(f(z)) $ by choosing a presentation of specializations for
the length of the subspace $Sp(z)$.

To prove $ l(z)\geq l(f(z)) $, we have two cases for the point $z$
from the assumption that $f$ is of Sp-type.

\textsl{\textbf{Case (i).}} Let $Sp(z)=f^{-1}(Sp(f(z)))$.

As $l(f(z))\leq l(Y)=\dim Y <\infty$, it is easily seen that there
exists some point $w\in Sp(f(z))$ such that
\begin{equation*}
l(f(z),w)=l(f(z))
\end{equation*}
by taking a presentation of specializations for the length of the subspace $%
Sp(f(z))$. Take a point $u\in Sp(z)$ such that $w=f(u)$. As we have
proved in \textsl{Step 1} that $f$ is length-preserving, we obtain
\begin{equation*}
l(f(z))=l(f(z),w)=l(z,u)\leq l(z).
\end{equation*}

\textsl{\textbf{Case (ii).}} Let $f^{-1}(Sp(f(z)))$ be
$Sp-$connected.

Prove $ l(f(z))\leq l(z). $ In deed, if $l(f(z))=0$, we have
\begin{equation*}
l(f(z))=0\leq l(z).
\end{equation*}

Let $l(f(z)) \geq 1$. Hypothesize that there is some point $v\in
Sp(f(z))$ such that
\begin{equation*}
l(f(z)) \geq l(f(z),v) \geq 1+l(z).
\end{equation*}

Take a point $ x\in f^{-1}(Sp(f(z))) $ with $f(x)=v$. As the points
$ x$ and $z$ are both contained in $f^{-1}(Sp(f(z)))$, it is seen
that either $z\rightarrow x$ or $x\rightarrow z$ is a specialization
from the assumption above.

If $z\rightarrow x$ is a specialization, we have $l(f(z),v)=l(z,x)$
since $f$ is length-preserving; then
\begin{equation*}
1+l(z)\leq l(f(z),f(x))=l(z,x)\leq l(z),
\end{equation*}
where there will be a contradiction.

If $x\rightarrow z$ is a specialization, we have a specialization
$f(x)\rightarrow f(z)$ by \emph{Lemma 1.5}; as the point $f(x)=v$ is
contained in the set $Sp(f(z))$, there is a generic specialization $
f(x)\leftrightarrow f(z) $. By \emph{Lemma 1.1} it is seen that
$v=f(z)$ holds. Then we have
\begin{equation*}
0=l(v,v)=l(f(z),v)\geq 1+l(z) \geq 1,
\end{equation*}
where we will obtain a contradiction.

Hence, we must have $ l(f(z))\leq l(z) $. This proves $ l(f(z))=l(z)
$ for any $z\in X$.

\textsl{\textbf{Step 3.}} Show $f$ is level-separated. Let $x_{1},x_{2}\in X$ be $Sp-$%
disconnected with $l\left( x_{1}\right) =l\left( x_{2}\right)$. We have $%
x_{1}\not =x_{2}$ by \emph{Lemma 1.1}.

Then $f\left( x_{1}\right) $ and $f\left( x_{2}\right) $ are $Sp-$%
disconnected. Otherwise, hypothesize that there is a specialization
$ f\left( x_{1}\right) \rightarrow f\left( x_{2}\right) $ in $Y$.
Consider the irreducible closed subsets
\begin{equation*}
Sp\left( f\left( x_{1}\right) \right) \supseteq Sp\left( f\left(
x_{2}\right) \right).
\end{equation*}

By \textsl{Step 2} we have
\begin{equation*}
l(f(x_{1}))=l(x_{1})=l(x_{2})=l(f(x_{2}));
\end{equation*}
then
\begin{equation*}
\dim Sp(f(x_{1}))=l(f(x_{1}))=l(f(x_{2}))=\dim Sp(f(x_{2})) < \infty.
\end{equation*}

It follows that
\begin{equation*}
Sp(f(x_{1}))=Sp(f(x_{2}))
\end{equation*}
holds.

By \emph{Remark 1.3} we have $ f(x_{1})=f(x_{2}) $ as generic points
of the irreducible closed set. As $f$ is injective, we must have
$x_{1}=x_{2}$, where there will be a contradiction to the assumption
above. This proves that $f\left( x_{1}\right) $ and $f\left(
x_{2}\right) $ are $Sp-$disconnected.

\textbf{(ii).} Prove the other half of the theorem. Let $X$ be
cat$\acute{e}$naire and let $f$ be length-preserving and
level-separated. We will prove that $f$ is injective.

It is seen that $\dim X=l(X)<\infty$. In deed, if $\dim X=\infty,$
we have $l(X)=\infty$ and then $\dim Y\geq l(Y)=\infty$ since $f$ is
length-preserving, which is in contradiction to the assumption.

Now fixed any $x,y\in X$. Let $\xi$ be the generic point of $%
X$.  In the following we will prove $f\left(
x\right) \not =f\left( y\right) $ if $x\not =y$.

There are three cases such as the following.

\textsl{\textbf{Case (i).}} Let $\dim X=0$.

We have $x=y$ and of course $f$ is injective.

\textsl{\textbf{Case (ii).}} Let $\dim X>0$ and let $x=\xi$ without
loss of generality.

If $x\not =y$, we have $y\neq\xi$ and then $x\rightarrow y$ is a
specialization in $X$. It follows that
$
l\left( x,y\right) >0
$
holds. As $f$ is length-preserving, We have
\begin{equation*}
l\left( f\left( x\right) ,f\left( y\right) \right) =l\left( x,y\right) >0.
\end{equation*}
Hence, $f\left( x\right) \neq f\left( y\right)$.

\textsl{\textbf{Case (iii).}} Let $\dim X>0,$ $x\neq\xi$ and $y\neq
\xi $.

As $l(X)=\dim X<\infty,$ for any $z\in X$ we have
\begin{equation*}
l\left( z\right) \leq l\left( X\right) <\infty.
\end{equation*}

There are several subcases such as the following.

\textsl{\textbf{Subcase (iii$_{a}$).}} Let $l\left( x\right)
=l\left( y\right) $ and let $x,y$ be Sp-connected.

Assume $y\in Sp\left( x\right) $ without loss of generality. We have
$$Sp(x)\supsetneqq Sp(y);$$ $$\dim (Sp(x))=l\left( x\right) =l\left(
y\right) =\dim (Sp(y)).$$ Then $Sp(x)=Sp(y)$ and hence $x=y$. So we
have $ f\left( x\right) =f\left( y\right)$. Such a subcase is
trivial.

\textsl{\textbf{Subcase (iii$_{b}$).}} Let $l\left( x\right)
=l\left( y\right) $ and let $x,y$ be Sp-disconnected.

We have $x \not= y$. As $f$ is level-separated, we have $f\left(
y\right) \not \in Sp\left( f\left( x\right) \right) $; hence, $$
f\left( x\right) \not =f\left( y\right).$$

\textsl{\textbf{Subcase (iii$_{c}$).}} Let $l\left( x\right)
>l\left( y\right) $ without loss of generality and let $x,y$ be Sp-connected.

We have $x\not= y$. It is clear that only $x \rightarrow y$ is a
specialization. As $f$ is length-preserving, We have
\begin{equation*}
l\left( f\left( x\right) ,f\left( y\right) \right) =l\left( x,y\right) >0.
\end{equation*}
Hence,
\begin{equation*}
f\left( x\right) \neq f\left( y\right) .
\end{equation*}

\textsl{\textbf{Subcase (iii$_{d}$).}} Let $l\left( x\right)
>l\left( y\right) $ without loss of generality and let $x,y$ be Sp-disconnected.

We have $x\not= y$. Prove $f(x)\not =f(y)$.

In fact, choose a presentation $\Gamma (x,u)$ of specializations
\begin{equation*}
x\rightarrow \cdots\rightarrow x_{0}\rightarrow \cdots \rightarrow u
\end{equation*}
in $X$ for the length $l\left( x\right)< \infty $. That is,
$l(x,u)=l(x)$. Here we have some point $%
x_{0}\in Sp\left( x\right) $ such that
\begin{equation*}
l\left( x_{0}\right) =l\left(y\right) <\infty
\end{equation*}
since $l(x)$ and $l(y)$ are nonnegative integers. As $X$ is
cat$\acute{e}$naire, by \emph{Claim 2.9} below we choose $x_0$ to be
the point such that $$l(x_0)=l(x_0,u).$$

It follows that we have
\begin{equation*}
l\left( x,x_{0}\right)=l(x,u)-l(x_{0},u)=l(x)-l(y) >0.
\end{equation*}

Then $x_{0}\neq y$. It is seen that $x_{0}$ and ${y}$ are
$Sp-$disconnected. Otherwise, if $x_0\rightarrow y$ is a
specialization, it is seen that $x\rightarrow y$ is a
specialization, which will be in contradiction to the assumption in
this subcase. If $y\rightarrow x_0$ is a specialization, we have
$$l(x_0)=l(y)=l(y,x_0)+l(x_0)\geq 1+l(x_0)$$ since $X$ is
cat$\acute{e}$naire, where there will be a contradiction.

Thus, $x_{0}$ and ${y}$ are $Sp-$disconnected and of the same length. As $%
f$ is level-separated, we have
\begin{equation*}
f\left( x_{0}\right) \neq f\left( y\right) .
\end{equation*}

By \emph{Claim 2.9} we have
\begin{equation*}
l\left( \xi ,x_{0}\right)=\dim X - l(x_0)=\dim X - l(y) =l\left( \xi,y\right);
\end{equation*}
then
\begin{equation*}
l\left( \xi,y\right) =l\left( \xi,x_{0}\right) =l\left( \xi,x\right) +l\left( x,x_{0}\right).
\end{equation*}

 As $f$ is length-preserving, we have
\begin{equation*}
l\left( f\left( x\right) ,f\left( x_{0}\right) \right) =l\left(
x,x_{0}\right) >0;
\end{equation*}
hence,
\begin{equation*}
\begin{array}{l}
l\left( f\left( \xi\right) ,f\left( y\right) \right) \\
=l\left( \xi,y\right) \\
=l\left( \xi,x_0\right) \\
=l\left( f\left( \xi\right) ,f\left( x_{0}\right) \right) \\
=l\left( f\left( \xi\right) ,f\left( x\right) \right) +l\left( f\left(
x\right) ,f\left( x_{0}\right) \right) \\
\leq l(Y)\\
=\dim Y\\
 <\infty.%
\end{array}%
\end{equation*}

We must have
\begin{equation*}
f\left( x\right) \neq f\left( y\right) .
\end{equation*}
Otherwise, if $f\left( x\right) =f\left( y\right)$, we have
\begin{equation*}
\begin{array}{l}
l\left( f\left( \xi\right) ,f\left( x\right) \right) \\
=l\left( f\left( \xi\right) ,f\left( y\right) \right) \\
=l\left( f\left( \xi\right) ,f\left( x\right) \right) +l\left( f\left(
x\right) ,f\left( x_{0}\right) \right) ;%
\end{array}%
\end{equation*}
then
\begin{equation*}
l(x,x_0)=l\left( f\left( x\right) ,f\left( x_{0}\right) \right) =0;
\end{equation*}
it follows that $x=x_0$ holds, where there will be a contradiction.
Therefore, $f$ is injective.

This completes the proof.
\end{proof}

\begin{claim}
Let $X$ be a cat\'{e}naire and irreducible scheme of finite
dimension. Let $\xi$ be the generic point of $X$.

$(i)$ We have $$l(x_0,x_r)=l(x_0,x_1)+l(x_1,x_2)+\cdots
+l(x_{r-1},x_r)$$ for any chain of specializations $x_0\rightarrow
x_1\rightarrow \cdots \rightarrow x_r$ in $X$.

$(ii)$ For any  point $x$ of $X$, we have
\begin{equation*}
l(\xi ,x)=\dim X -l(x).
\end{equation*}
In particular, we have
\begin{equation*}
l(\xi ,u)=\dim X
\end{equation*}
for any closed point $u$ of $X$.
\end{claim}

\begin{proof}
${(i)}$ It is immediate by induction on $r$. In fact, take any
specializations $$x\rightarrow y\rightarrow z$$ in $X$. We have
irreducible closed subsets
$$Sp(x)\supseteq Sp(y) \supseteq Sp(z).$$ As $X$ is a cat\'{e}naire
space$^{[1]}$,  we have
$$l(x,z)=l(x,y)+l(y,z).$$

${(ii)}$  Let $u$ be a closed point of $X$. From the property of
cat\'{e}naire spaces$^{[1]}$, we have
\begin{equation*}
l(\xi ,u)=l(X)=\dim X
\end{equation*}
by taking a presentation of specializations for the length $l(X)$.

By $(i)$ it is easily seen that \begin{equation*} l(\xi ,x)=\dim X
-l(x)
\end{equation*}
for any  point $x$ of $X$.
\end{proof}

\begin{corollary}
Let $f:X\rightarrow Y$ be a morphism of schemes and let $ Y$ be of
finite dimension. Then we have
\begin{equation*}
\left\Vert f\right\Vert =1
\end{equation*}%
if $f$ is injective and of Sp-type.
\end{corollary}

\begin{proof}
By \emph{Theorem 2.7} we have $ 0\leq \left\Vert f\right\Vert \leq
1$.  As $f$ is length-preserving by \emph{Theorem 2.8}, we have
$$l(x,y)=l(f(x),f(y))$$ for any specialization $x\rightarrow
y$ in $X$. Hence, $\left\Vert f\right\Vert =1$ holds.
\end{proof}

\begin{remark}
There are some concrete examples from commutative rings shows that
the local condition, the Sp-type,  can not be removed from the above
theorems.
\end{remark}

\begin{remark}
The topic on combinatorial norms of morphisms of schemes, discussed
above in the paper, can be applied to study the discrete Morse
theory on arithmetic schemes and Kontsevich's theory of graph
homology.
\end{remark}


\newpage

\begin{thebibliography}{9}

\bibitem{Comm} N. Bourbaki. Alg\`{e}bre Commutative. Chapitres 8 et
9.
Masson, Paris, 1983.

\bibitem{morse1} M. K. Chari $\&$ M. Joswig. Complexes of discrete Morse
functions. Discrete Math. 302 (2005), 39-51.

\bibitem{Forman1} R. Forman. A Discrete Morse Theory for Cell Complexes. Geometry, Topology
$\&$ Physics for Raoul Bott, S.T. Yau (ed.). International Press,
Boston, 1995, pp 112-125.

\bibitem{Forman2} R. Forman. Morse theory for cell complexes. Adv. in Math. 134 (1998), 90-145.

 \bibitem{Forman3} R. Forman. Witten-Morse Theory for Cell Complexes. Topology
37 (1998), 945-979.

\bibitem{Forman4} R. Forman. Combinatorial Vector Fields and Dynamical Systems. Math.
Zeit. 228 (1998), 629-681.

\bibitem{Forman5} R. Forman. Combinatorial differential topology and geometry. New Perspectives
in Algebraic Combinatorics. Cambridge Univ. Press, Cambridge, 1999,
pp 177-206.

\bibitem{Forman6} R. Forman. Morse Theory and Evasiveness. Combinatorica 20 (2000), 498-504.

\bibitem{grph} C. Godsil $\&$ G. Royle. Algebraic Graph Theory.
Springer 2001, New York.

\bibitem{EGA} A. Grothendieck $\&$ J. Dieudonn\'{e}. \'{E}l\'{e}ments de
G\'{e}oem\'{e}trie Alg\'{e}brique. vols I-IV, Pub. Math. de l'IHES,
1960-1967.

\bibitem{SGA1} A. Grothendieck $\&$ M. Raynaud. Rev\^{e}tements \'{E}tales et Groupe Fondamental
(SGA 1). Lecture Notes in Math., vol. 224. Springer, New York, 1971.

\bibitem{Hrtsh} R. Hartshorne. Algebraic Geometry. Springer, New York, 1977.

\bibitem{gh6} K. Igusa. Graph cohomology and Kontsevich cycles. Topology 43 (2004),
1469-1510.

\bibitem{grph} J Jonsson. Simplicial Complexes of Graphs. Lecture Notes in Math., vol
1928. Springer, New York, 2008.

\bibitem{morse2} M. Joswig $\&$ M. Pfetsch. Computing optimal discrete morse functions.
SIAM J. Discrete Math. 20 (2006), 11-25.

\bibitem{gh1} M. Kontsevich. Intersection theory on the moduli space of curves and
the matrix Airy function. Commun. Math. Phys. 147 (1992), 1-23.

\bibitem{gh2} M. Kontsevich. Formal (non)commutative symplectic geometry. The
Gelfand Mathematical Seminars (1990-1992). BirkhUauser, Boston 1993,
pp 173-187.

\bibitem{gh3} M. Kontsevich. Feynman diagrams and low-dimensional
topology. First European Congress of Mathematics, Vol. II (Paris,
1992). BirkhUauser, Basel, 1994, pp 97-121.

\bibitem{mum} D. Mumford. The Red Book of Varieties and Schemes. Lecture Notes
in Math., vol 1358. Springer-Verlag, Berlin, 1988.

\bibitem{Orlik} P. Orlik $\&$ V. Welker. Algebraic Combinatorics. Springer, New York, 2007.

\bibitem{sha} I. R. Shafarevich. Basic Algebraic Geometry. Second Edition.
Springer, Berlin, 1994.

\bibitem{Sh} P. K. Sharma. A note on lifting of chains of prime
ideals.
Journal of Pure and Applied Algebra, 192(2004), 287-291.

\bibitem{morse4} E. Sk\"{o}ldberg. Combinatorial discrete Morse theory from an
algebraic viewpoint. Trans. Amer. Math. Soc. 358 (2006), 115-129.

\bibitem{morse5} R.P. Stanley. Combinatorics and commutative algebra. Second
edition. Progress in Mathematics, 41. Birkh\"{a}user Boston, Inc.,
Boston, MA, 1996

\bibitem{Weil} A. Weil. Foundations of Algebraic Geometry. Amer. Math.
Society, New York City, 1946.
\end{thebibliography}
\end{document}